\documentclass[11pt]{article}
\usepackage[a4paper]{geometry}
\usepackage{amssymb, amsmath, amsthm, amsfonts, amscd}
\usepackage[english]{babel}
\usepackage{graphicx}
\usepackage{caption}
\usepackage{subcaption}
\usepackage{color,soul}
\usepackage{booktabs,dcolumn}                     
\usepackage[figureposition=bottom]{caption}       
\usepackage[usenames,dvipsnames,svgnames,table]{xcolor}
\usepackage{authblk}
\usepackage[title]{appendix}

\newtheorem{theorem}{Theorem}

\newtheorem*{theorem*}{Theorem}
\newtheorem{lemma}{Lemma}

\newtheorem*{assumption*}{Assumption}

\newtheorem{remark}{Remark}

\numberwithin{equation}{section}
\numberwithin{theorem}{section}
\numberwithin{remark}{section}
\numberwithin{prop}{section}
\numberwithin{corollary}{section}
\numberwithin{lemma}{section}

\title{Explicit asymptotics for certain single and double exponential sums}
\author[1]{K. Kalimeris} 
\author[1,2]{A. S. Fokas}
\affil[1]{\footnotesize Department of Applied Mathematics and Theoretical Physics, University of Cambridge, CB3 0WA, UK,}
\affil[2]{Viterbi School of Engineering, University of Southern California, Los Angeles, California, 90089-2560, USA. }

\begin{document}

\maketitle

\begin{abstract}
By combining classical techniques together with two novel asymptotic identities contained in \cite{FL}, we analyse certain single sums of Riemann-zeta type. In addition, we analyse Euler-Zagier double exponential sums for particular values of $Re\{u\}$ and $Re\{v\}$ and for a variety of sets of summation, as well as particular cases of Mordell-Tornheim double sums. Some of these results are used in \cite{Fs} where a novel approach to the Lindel{\"o}f hypothesis is presented.
\end{abstract}

\noindent
{\it Keywords}:


\section{Introduction}

The Riemann hypothesis, perhaps the most celebrated open problem in the history of mathematics, is valid iff $\zeta(s)\neq 0, \ s=\sigma+i t, \ 0<\sigma<\frac{1}{2}, \ t>0,$ where $\zeta(s)$ denotes the Riemann zeta function. This hypothesis has been verified numerically for $t$ up to order $O\left(10^{13}\right)$, thus the basic problem reduces to the relevant proof for large $t$. The study of the large $t$ asymptotics of the Riemann zeta function, which has a long and illustrious history, is deeply related with the Lindel{\"o}f hypothesis.

The large $t$ asymptotics to all orders of $\zeta(s)$ is studied in \cite{FL}. Furthermore, a novel approach to the Lindel{\"o}f hypothesis is presented in \cite{Fs}. The analysis of some of the formulae appearing in \cite{Fs} requires the analysis of certain single and double exponential sums. 
Here, motivated by the appearance of the above single and double Riemann-zeta type sums in connection with the Lindel\"of hypothesis, we revisit such sums. In particular, in section 2 we revisit a novel identity derived in \cite{FL} and also, using the results of \cite{FL}, we present a variant of the above identity. These two identities, used by themselves or in combination with classical techniques \cite{T}, allow us to derive several estimates in a simpler way than using only the classical techniques. In section 2 we also derive some estimates for certain specific Riemann-zeta type single sums; these sums arise in sections 3 and 5 as a result of using the identities discussed in section 2.1 for estimating double Riemann-type sums. In section 3 we derive some simple estimates for double Riemann-zeta type exponential sums, we review some well known estimates for the Euler-Zagier sums defined on the critical strip $0\leq \sigma \leq 1$, and establish a connection between these two types of sums. Some of the results of this section are derived via the results of section 2. In section 4 we provide sharp estimates for particular cases of Euler-Zagier and Mordell-Tornheim sums. 
 In section 5, we derive estimates for two types of double exponential sums, denoted by $S_1$ and $S_2$ which involve ``smal" sets.  The analysis of $S_1$ is also based on the results of section 2 and illustrates the fact that double sums involving ``small'' sets can be studied via the variant of the identities of \cite{FL} presented in section 2.1, in a simpler way than using classical estimates. Furthermore, and more importantly, this novel approach yields sharp results. This fact is further demonstrated in the analysis of $S_2$: this sum can be studied directly via classical estimates or even via ``rough" estimates, however the above novel approach yields significantly sharper results; details are given in section 5.


\vspace{3mm}
\underline{Notation}
\vspace{1mm}

$ \qquad [A] =$ integer part of $A$.

\section{Asymptotic estimates and identities of certain single exponential sums}

In this section we analyse sums of the type
\begin{equation}\label{sf0}
\sum_{m=A(t)}^{B(t)} e^{  if(m) }, \qquad 1\leq A(t) < B(t),
\end{equation}
for the following three particular cases
\begin{equation} \label{sf1}
f(m)=  t \ln{\left( 1 +\frac{t}{m}\right)}, \quad t>0, \quad m\in\mathbb{Z}^{+},
\end{equation}
\begin{equation} \label{sf2}
f(m)=  t \ln{\left( 1 +\frac{m}{t}\right)}, \quad t>0, \quad m\in\mathbb{Z}^{+}.
\end{equation}
and
\begin{equation} \label{sf3}
f(m)=  t \ln{m}, \quad t>0, \quad m\in\mathbb{Z}^{+}.
\end{equation}

The case \eqref{sf3} corresponds to the classical exponential sums related to Riemann zeta function. In this case,  partial summation and the Phragm\'en-Lindel{\"o}f convexity principle (PL)  (known also as Lindel{\"o}f's theorem) implies
\begin{equation} \label{est0}
\sum_{m=1}^{[t]} \dfrac{1}{m^\sigma}e^{  it \ln{m} } =\sum_{m=1}^{[t]}m^{-\sigma+it} = \begin{cases}
O \left(  t^{\frac{1}{2}-\frac{2}{3}\sigma} \ln{t} \right), & 0\leq\sigma\leq \frac{1}{2},\\
O \left(  t^{\frac{1}{3}-\frac{1}{3}\sigma} \ln{t}  \right), & \frac{1}{2} < \sigma< 1.
\end{cases}
\end{equation}

The exponents $\ \frac{1}{2}-\frac{2}{3}\sigma \ $ and $\ \frac{1}{3}-\frac{1}{3}\sigma \ $ have been improved only slightly in the last 100 years with the best current result due to Bourgain \cite{B}.

\subsection{Two useful asymptotic identities}

In what follows we present a slight variant of two useful asymptotic identities derived in \cite{FL}.

Below we study sums of the type \eqref{sf0} with $f(m)$ given by \eqref{sf3}: the cases (i) and (ii)  correspond to  $t \leq A(t) < B(t)$ and $ A(t) < B(t) = \mathcal{O}(t)$, respectively.

\begin{lemma}\label{l2.1}
\begin{itemize}

 \item[(i)]
\begin{align}  \label{FL-eq}
\sum_{n=[t]+1}^{\left[\frac{\eta}{2\pi }\right]}\frac{1}{n^{s}}&=\frac{1}{1-s}\left(\frac{\eta }{2\pi}\right)^{1-s}+O\left(\frac{1}{t^{\sigma }}\right), 
\qquad s=\sigma +it, \quad  t<\frac{\eta}{2\pi} <\infty , \notag \\
&0\leq\sigma<1, \quad t\to \infty .
\end{align}

\item[(ii)]  
\begin{multline} \label{FR}
\sum_{n=\left[ \frac{t}{\eta_2}\right] +1 }^{\left[ \frac{t}{\eta_1}\right]} \frac{1}{n^s} = \chi(s) \sum_{n=\left[ \frac{\eta_1}{2\pi}\right] +1 }^{\left[ \frac{\eta_2}{2\pi}\right]}  \frac{1}{n^{1-s}} + E(\sigma,t,\eta_2) -  E(\sigma,t,\eta_1), \quad t\to\infty,\\
0<\sigma<1, \quad \varepsilon<\eta_1<\eta_2<\sqrt{t}, \quad \varepsilon>0; \quad {\operatorname{dist}(\eta_j, 2\pi \mathbb{Z})}>\varepsilon, \quad j=1,2,
\end{multline}
where as $ t\to \infty$,
\begin{subequations} \label{seventhirtysix}
\begin{multline} \label{seventhirtysixa}
E(\sigma,t,\eta) = e^{i\gamma} \left(  \frac{\eta}{t} \right) ^s \left(  1+ O\left(  \frac{1}{t} \right) \right) \\
\times \left\{ \frac{1}{\alpha} + \frac{i}{2\alpha^3} \frac{\eta^2}{t} \left[  \frac{\alpha^2}{\eta^2}(\beta^2+\sigma-1) -2 \frac{\alpha\beta}{\eta} - \alpha +2 \right]  \right\} \\
+ \begin{cases}
      O\left( \frac{\eta}{t}  \right), & \varepsilon<\eta<t^{\frac{1}{3}}, \quad 3\eta^3<\alpha t, \\
      O\left( e^{-\frac{\alpha t}{\eta^2}}+  \frac{\eta^4}{t^2}  \right), & t^{\frac{1}{3}}<\eta<\sqrt{t}, \quad 3\eta^2<\alpha t, \\
   \end{cases}
\end{multline}
with $\alpha, \ \beta$ and $\gamma$,  defined by
\begin{equation} \label{seventhirtysixb}
\alpha(\eta) = 1 - e^{-i\eta}, \quad \eta>0,
\end{equation}
\begin{equation} \label{seventhirtysixc}
\beta(\sigma, t, \eta) = t- \eta \left[  \frac{t}{\eta} \right] - i(\sigma-1), \quad 0<\sigma<1, \quad t>0, \quad \eta>0,
\end{equation}
\begin{equation}
\gamma(t,\eta)=t-\eta-\eta \left[  \frac{t}{\eta}\right], \quad t>0, \quad \eta>0,
\end{equation}
and $\chi(s)$ defined by
\begin{equation} \label{seventhirtysixd}
\chi(s) = \frac{(2\pi)^s}{\pi} \sin{\left( \frac{\pi s}{2} \right)} \Gamma(1-s), \quad s\in\mathbb{C}.
\end{equation}
\end{subequations}

The above results are valid uniformly with respect to $\eta $ and $\sigma $.

\end{itemize}
\end{lemma}

\begin{proof}
\begin{itemize}
\item[(i)] Equation \eqref{FL-eq} is given by equation (1.9) of \cite{FL}, with
 $$\eta_1=2\pi t> (1+\epsilon)t \qquad \text{ and } \qquad  \eta_2=\eta>\eta_1,$$
 for some $\epsilon>0$.
\item[(ii)] Regarding \eqref{FR}, we first recall equation (4.2) of \cite{FL}:
\begin{multline} \label{seventhirtyseven}
\zeta(s) = \sum_{n=1}^{\left[  \frac{t}{\eta} \right]} \frac{1}{n^s} + \chi(s) \sum_{n=1}^{\left[  \frac{\eta}{2\pi} \right]} \frac{1}{n^{1-s}} \\
+ ie^{-\frac{i\pi s}{2}} \frac{\Gamma(1-s)}{\sqrt{2\pi}}e^{-i \left( \left[\frac{t}{\eta} \right] + 1 \right) \eta} e^{-\frac{i\pi}{4}} \frac{\eta^s}{\sqrt{t}} \\
\times \left\{ \frac{1}{\alpha} + \frac{i}{2\alpha^3} \frac{\eta^2}{t} \left[  \frac{\alpha^2}{\eta^2}(\beta^2+\sigma-1) -2 \frac{\alpha\beta}{\eta} - \alpha +2 \right]  \right\} \\
+ e^{-i\pi s} \Gamma(1-s) e^{-\frac{\pi t}{2}} \eta^{\sigma-1} \begin{cases}
      O\left( \frac{\eta}{t}  \right), & \varepsilon<\eta<t^{\frac{1}{3}}, \quad 3\eta^3<\alpha t, \\
      O\left( e^{-\frac{\alpha t}{\eta^2}}+  \frac{\eta^4}{t^2}  \right), & t^{\frac{1}{3}}<\eta<\sqrt{t}, \quad 3\eta^2<\alpha t, \\
   \end{cases}\\
{\operatorname{dist}(\eta, 2\pi \mathbb{Z})}>\varepsilon, \quad 0<\sigma<1, \quad t\to\infty.
\end{multline}
The asymptotic formula 
\begin{equation*} \label{2.19b}
\Gamma(\sigma-i\xi) = \sqrt{2\pi}\xi^{\sigma-\frac{1}{2}}e^{-\frac{\pi \xi}{2}}e^{\frac{i\pi}{4}}e^{i\xi}\xi^{-i\xi}
e^{-\frac{i\pi \sigma}{2}} \left[ 1 + O\left(\frac{1}{\xi}\right) \right], \quad \xi \to\infty,
\end{equation*}
which is proven in the Appendix A of \cite{FL}, implies
\begin{equation*}
e^{-\frac{i\pi s}{2}} \Gamma(1-s) = \sqrt{2\pi} e^{it}t^{\frac{1}{2}-s} e^{-\frac{i\pi}{4}} \left(  1 + O\left( \frac{1}{t}  \right) \right), \quad t\to\infty.
\end{equation*}
Thus, equation \eqref{seventhirtyseven} becomes
\begin{equation*}
\zeta(s) = \sum_{n=1}^{\left[  \frac{t}{\eta} \right]} \frac{1}{n^s} + \chi(s) \sum_{n=1}^{\left[  \frac{\eta}{2\pi} \right]} \frac{1}{n^{1-s}} + E(\sigma,t,\eta),
\end{equation*}
with $E$ defined in \eqref{seventhirtysixa}. Evaluating this expression for two different values of $\eta$, namely $\eta_1$ and $\eta_2$, where  $0<\varepsilon<\eta_1<\eta_2<\sqrt{t}, \ $ and subtracting the resulting equations we obtain \eqref{FR}.
\end{itemize}
\end{proof}

\begin{remark}
Equation \eqref{FL-eq} is a special form of the general case
\begin{align*}
\sum_{n=[\tau]+1}^{\left[\frac{\eta}{2\pi }\right]}\frac{1}{n^{s}}=\frac{1}{1-s}\left(\frac{\eta }{2\pi}\right)^{1-s}+O\left(\frac{1}{t^{\sigma }}\right), 
\qquad0\leq\sigma<1, \quad t\to \infty,
\end{align*}
where $\tau=O(t)$, provided that $\tau > (1+\epsilon)\frac{t}{2\pi},$ for some $\epsilon>0$.

In connection with equation \eqref{FR}, the definitions of $\alpha,\beta, \gamma$ yield the following bounds:
 $$|\alpha|>\varepsilon,\qquad 0<|\beta|<\eta+1, \qquad 0<|\gamma|\leq \eta.$$ 
\end{remark}

\subsection{Asymptotic estimates of single sums}

In the following two Lemmas we consider \eqref{sf0}  and set $A(t)=1$, $B(t)=[t]$, with $f(m)$ given by \eqref{sf1} and \eqref{sf2}.

\begin{lemma}\label{l2.2}
Let $f(m)$ be defined by \eqref{sf1}. Then
\begin{equation} \label{est1}
\sum_{m=1}^{[t]} \dfrac{1}{m^\sigma}e^{  if(m) } = \begin{cases}
O \left(  t^{\frac{1}{2}-\frac{2}{3}\sigma} \ln{t} \right), & 0\leq\sigma\leq \frac{1}{2},\\
O \left(  t^{\frac{1}{3}-\frac{1}{3}\sigma} \ln{t}  \right), & \frac{1}{2} < \sigma< 1,
\end{cases} \qquad t\to \infty.
\end{equation}
\end{lemma}

\begin{proof}

 Observe that the $k$-th derivative of $f(x)$ satisfies
\begin{equation*}
f^{(k)}(x) = (-1)^{k-1} (k-1)! ~ t \left[ \frac{1}{(x+t)^k} - \frac{1}{x^k}  \right].
\end{equation*}
Thus,
\begin{equation*}
\left|f^{(k)}(x)\right| = (k-1)! ~ \frac{t}{x^k}~ C(x,t;k),
\end{equation*}
where $C(x,t;k)$ is defined by
\begin{equation*}
C(x,t;k) = \frac{1 + \sum\limits_{n=1}^{k-1} \binom {k} {n} \left( \frac {x} {t} \right)^n }{1 + \sum\limits_{n=1}^{k} \binom {k} {n} \left( \frac{x} {t} \right)^n }.
\end{equation*}
The function $C(x,t;k)$ is bounded, namely,
\begin{equation} \label{B.3}
1-2^{-k} <  C(x,t;k) < 1, \quad \text{for} \quad 1<x<t.
\end{equation}
Hence, we can use Theorem 5.14 of \cite{T} with
\begin{equation*}
\lambda_k = \frac{(k-1)!}{2\pi} \frac{t}{(2\alpha)^k} (1-2^{-k}),
\end{equation*}
and
\begin{equation*}
h =  \frac{2^k}{1-2^{-k}}, \quad k\ge 2.
\end{equation*}
Setting $A(t)=1$ and $B(t)=[t]$ in \eqref{sf0}, we define
\begin{equation}
D(\sigma,t)=\sum_{m=1}^{[t]} \frac{1}{m^{\sigma} }e^{if(m)},
\end{equation}
with $f(m)$ given by \eqref{sf1}.

For $k=2$, by applying the partial summation technique, we obtain
\begin{equation} \label{B.4a}
D(0,t) = O \left( t^{\frac{1}{2}} \ln{t} \right), \qquad  t\to \infty.
\end{equation}

Similarly, for $k=3$, we obtain
\begin{equation} \label{B.4}
D\left( \frac{1}{2} , t \right)
= O \left( t^{\frac{1}{6}} \ln{t} \right), \qquad  t\to \infty.
\end{equation}

\noindent We also note the following:

\begin{enumerate}
\item The Phragm\'en-Lindel{\"o}f convexity principle (PL) implies
\begin{equation*}
D(\sigma ,t) = \begin{cases}
O \left(  t^{\frac{1}{2}-\frac{2}{3}\sigma} \ln{t} \right), & 0\leq\sigma\leq \frac{1}{2},\\
O \left(  t^{\frac{1}{3}-\frac{1}{3}\sigma} \ln{t}  \right), & \frac{1}{2} < \sigma< 1,
\end{cases} 
\qquad  t\to \infty,
\end{equation*}
which gives \eqref{est1}.
\item If
\begin{equation*}
\sigma=\sigma(\ell)=1-\frac{\ell}{2L-2}, \quad L=2^{\ell -1}, \quad \ell\ge 3, \quad \ell \in \mathbb{N},
\end{equation*}
then for $\sigma =\sigma(\ell)\ge \frac{1}{2}$, we find
\begin{equation}  \label{B.5}
D(\sigma, t) = O \left(  t^{\frac{1}{2L-2}} \ln{t} \right), \qquad  t\to \infty.
\end{equation}

\item The PL principle allows the extension of the above result for the case of $\sigma \in (\sigma(\ell), \sigma(\ell+1))$ and $0\le\sigma\le \frac{1}{2}$.

\item Let $D_{\delta}$ be defined by
\begin{equation}  \label{B.6}
D_{\delta}(\sigma,t) = \sum\limits_{m=1}^{\left[ t^{\delta}\right]} \frac{1}{m^{\sigma}}e^{if(m)},
\end{equation}
where $\delta$ is a sufficiently small, positive constant.

By applying Theorem 5.14 of \cite{T}, for $k= \left[\frac{1}{\delta}\right] +1$, it can be shown that
\begin{equation}  \label{B.7}
D_{\delta}(\sigma,t) = O \left(  t^{(1-\sigma)\delta} \right), \qquad  t\to \infty.
\end{equation}
 However, we do not present the details of this proof here, since \eqref{B.7} can be obtained by the following simple estimate:
$$\left|\sum_{m=1}^{\left[ t^{\delta}\right]} \frac{1}{m^{\sigma}}e^{if(m)}\right|\leq \int_{1}^{t^{\delta}} \frac{1}{x^{\sigma}}dx=O \left(  t^{(1-\sigma)\delta} \right), \qquad  t\to \infty.$$

\end{enumerate}
\end{proof}

\begin{lemma}\label{l2.3}
Let $f(m)$ be defined by \eqref{sf2}. Then
\begin{equation}\label{est2}
\sum_{m=1}^{[t]} \dfrac{1}{m^\sigma} e^{  if(m) }=O(1), \quad \sigma\geq 0, \qquad  t\to \infty.
\end{equation}
\end{lemma}

\begin{proof}

We follow the steps of the analysis in \cite{T} and we observe that in this procedure the upper and lower bounds of the term $\big| f^{(k)}(x) \big|$ are independent of $x$. Indeed, we have  $$\big| f^{(k)}(x) \big|=(k-1)!\dfrac{t}{(x+t)^k},$$ and using that $0<x\leq t$ we get the conditions of  Theorem 5.13 in \cite{T}, i.e.$$\lambda_k\leq \big| f^{(k)}(x) \big| \leq h \lambda_k,$$ with $\lambda_k=\frac{(k-1)!}{2^k} t^{1-k} \ $ and $h=2^{k}$, for $k\geq 2$.

Thus, we may obtain the optimal estimates for the relevant sums, namely the sums $$\sum_{m=1}^{[t]} \frac{1}{m^{\sigma}}e^{if(m)}.$$

However, it is more efficient to use a different approach, based on Lemma 4.8 of \cite{T}. Indeed, it is straightforward to observe that $\displaystyle f'(x) =\dfrac{1}{1+\frac{x}{t}}$ is monotonic and also $f'(x)$ satisfies $\ \frac{1}{2} \leq \big| f'(x) \big|<1.$ Thus, the above Lemma yields 
$$\sum_{m=1}^{[t]} e^{  if(m) }=\int_{1}^{t} e^{  if(x) } dx + O(1), \qquad  t\to \infty.$$
The integral in the rhs of the above equation gives the contribution $$\dfrac{ (2t)^{1+i t}-(t+1)^{1+i t}}{1+i t}=-i \dfrac{2^{1+i t}-\left(1+\frac{1}{t}\right)^{1+i t}}{1-\frac{i}{t}}=O(1), \qquad  t\to \infty.$$
Therefore, the estimate \eqref{est2} holds for $\sigma=0$. 

The above analysis gives $$\sum_{m=a}^{b} e^{  if(m) } =  O(1), \qquad \text{for all } \ 1\leq a < b \leq t, \qquad  t\to \infty.$$
Hence, we apply the partial summation technique, with $m\geq 1$, i.e. $m^{-\sigma}\leq 1, \ \sigma>0,$ and we obtain \eqref{est2}.


\end{proof}

\section{Double zeta functions and Euler-Zagier double sums}

In this section we analyse the double zeta functions in the critical strip, namely the case that the real part of the exponents is in the interval $(0,1)$.

\subsection{Simple estimates for double exponential sums}

Letting $s=\sigma+it, \ \sigma\in(0,1)$, we estimate the double sums appearing of the form
\begin{equation}\label{dss1}
  \sum_{m=1}^{[t]}  \sum_{n=1}^{[t]}\frac{1}{m^s n^{\bar{s}}} .
\end{equation}

\begin{lemma}\label{l3.1}
The following estimate for \eqref{dss1} is valid:
\begin{equation}\label{dssest0}
  \sum_{m=1}^{[t]}  \sum_{n=1}^{[t]}\frac{1}{m^s n^{\bar{s}}}   =\begin{cases}
O \left(  t^{\frac{3}{2}-\frac{5}{3}\sigma} \ln{t} \right), & 0\leq\sigma\leq \frac{1}{2},\\
O \left(  t^{\frac{4}{3}-\frac{4}{3}\sigma}  \ln{t}  \right), & \frac{1}{2} < \sigma< 1,
\end{cases}
\qquad  t\to \infty.
\end{equation}  
\end{lemma}
\begin{proof}

First, we will use the following ``crude" estimates:
\begin{equation}\label{dssest1}
  \Bigg| \sum_{m=1}^{[t]}  \sum_{n=1}^{[t]}\frac{1}{m^s n^{\bar{s}}}  \Bigg|  \leq \int_1^t \int_1^t\dfrac{1}{x^\sigma} \dfrac{1}{y^\sigma} dxdy =O\left(t^{2-2\sigma}\right), \qquad  t\to \infty.
\end{equation}  

By employing techniques developed in \cite{T} it is possible to improve the estimates of \eqref{dss1}. Observing that 
\begin{equation}\label{dss2}
  \Bigg| \sum_{m=1}^{[t]}  \sum_{n=1}^{[t]} \frac{1}{m^s n^{\bar{s}}}  \Bigg|   \leq  \sum_{m=1}^{[t]}   \Bigg| \sum_{n=1}^{[t]} \frac{1}{n^{\bar{s}}} \Bigg| \frac{1}{m^{\sigma}}, 
\end{equation}  
and using the rough estimate 
\begin{equation}
\sum_{n=1}^{[t]} \frac{1}{n^{\bar{s}}} = O\left(t^{\frac{1}{2}-\frac{1}{2}\sigma} \ln t\right), \qquad  t\to \infty,
\end{equation}
we can improve the estimates of \eqref{dssest1} as follows:
\begin{equation}\label{dssest2}
  \sum_{m=1}^{[t]}  \sum_{n=1}^{[t]}\frac{1}{m^s n^{\bar{s}}}   =O\left(t^{1-\sigma} t^{\frac{1}{2}-\frac{1}{2}\sigma} \ln t\right)=O\left( t^{\frac{3}{2}-\frac{3}{2}\sigma} \ln t\right), \qquad  t\to \infty.
\end{equation}  

Further improvement of \eqref{dssest1} is obtained by employing \eqref{est0}, thus
\begin{equation*}\label{dssest3}
  \sum_{m=1}^{[t]}  \sum_{n=1}^{[t]}\frac{1}{m^s n^{\bar{s}}}   =O\left(t^{1-\sigma}\right)\times\begin{cases}
O \left(  t^{\frac{1}{2}-\frac{2}{3}\sigma} \ln{t} \right), & 0\leq\sigma\leq \frac{1}{2},\\
O \left(  t^{\frac{1}{3}-\frac{1}{3}\sigma}  \ln{t}  \right), & \frac{1}{2} < \sigma< 1,
\end{cases}
\qquad  t\to \infty,
\end{equation*}  
which  yields \eqref{dssest0}.
\end{proof}

\begin{remark}
The above improvement of the estimates becomes clearer for $\sigma=\frac{1}{2}$, where using \eqref{dssest1}, \eqref{dssest2} and \eqref{dssest0}, we obtain as $  t\to \infty$ the estimates $O(t)$, $O\left(t^\frac{3}{4}\ln t \right)$ and $O\left(t^\frac{2}{3} \ln t \right)$, respectively.
\end{remark}

\subsection{Estimates of Euler-Zagier sums}

In what follows we first review the estimates of the Euler-Zagier double sums as they were obtained in \cite{KT}, where techniques from \cite{K} and \cite{T2} were extensively used. A special case of Theorem 1.1 in \cite{KT} reads as follows:
\begin{theorem*}[\textbf{1.1} in \textbf{\cite{KT}}]
Let $s_j=\sigma_j+i t$, with $0\leq \sigma_j<1, \ j=1,2$. Then the following estimates are valid as $ t\to \infty$:
\begin{equation}
\sum_{1\leq m<n} \dfrac{1}{m^{s_1}}\dfrac{1}{n^{s_2}}=\begin{cases} O\left( t^{1-\frac{2}{3}(\sigma_1+\sigma_2)} (\ln t)^2\right), & 0\leq \sigma_1\leq \frac{1}{2}, \quad 0\leq \sigma_2\leq \frac{1}{2}, \\
O\left( t^{\frac{5}{6}-\frac{1}{3}(\sigma_1+2\sigma_2)} (\ln t)^3\right), & \frac{1}{2}< \sigma_1< 1, \quad 0\leq \sigma_2\leq \frac{1}{2}, \\
O\left( t^{\frac{5}{6}-\frac{1}{3}(2\sigma_1+\sigma_2)} (\ln t)^3\right), & 0\leq \sigma_1\leq \frac{1}{2}, \quad \frac{1}{2} < \sigma_2 <1, \\
O\left( t^{\frac{2}{3}-\frac{1}{3}(\sigma_1+\sigma_2)} (\ln t)^4\right), & \frac{1}{2} < \sigma_1 <1, \quad \frac{1}{2} < \sigma_2 <1. \end{cases}
\end{equation}
\end{theorem*}

As a corollary of the above we obtain the analogue of Corollary 1.2 in \cite{KT}, namely, as $t\to \infty$, we have the following:
\begin{align*}
&\sum_{1\leq m<n} \dfrac{1}{m^{it}}\dfrac{1}{n^{it}}=O\left( t (\ln t)^2\right), \\
&\sum_{1\leq m<n} \dfrac{1}{m^{it}}\dfrac{1}{n^{\frac{1}{2}+it}}=O\left( t^\frac{2}{3} (\ln t)^2\right), \\
&\sum_{1\leq m<n} \dfrac{1}{m^{\frac{1}{2}+it}}\dfrac{1}{n^{\frac{1}{2}+it}}=O\left( t^\frac{1}{3} (\ln t)^2\right).
\end{align*}
 
 The above results provide a `sharp' generalisation for double sums of the classical result of \cite{T}, as this is reviewed in  \eqref{est0}. In this sense, the above estimates improve significantly the analogous results of \cite{IM}.

\subsection{Relations between double exponential sums}

The results of subsections 3.1 and 3.2 suggest a connection between the double zeta function and the Euler-Zagier sums. Actually, the following exact relation between the Euler-Zagier sum and the leading asymptotic representation of $|\zeta|^2$ is valid:
\begin{align}\label{sevensix}
2\Re&\left\{ 
\sum_{m_{1}=1}^{[t]}\sum_{m_{2}=1}^{[t]}\frac{1}{m_{2}^{\bar{s}}(m_{1}+m_{2})^s } \right\} -
\left( \sum_{m=1}^{[t]}\frac{1}{m^{s}} \right) \left( \sum_{m=1}^{[t]}\frac{1}{m^{\bar{s} }} \right)\nonumber \\ 
&=-\sum_{m=1}^{[t]}\frac{1}{m^{2\sigma }}
+2\Re \left\{ \sum_{m=1}^{[t]}\sum_{n=[t]+1}^{[t]+m}\frac{1}{m^{\bar{s} }n^{s}} \right\}, 
\qquad s=\sigma + i t \in\mathbb{C}.
\end{align}

In order to establish this connection, we prove Lemma \ref{l7.1}. Indeed, equation (\ref{sevensix}) is a special case of Lemma \ref{l7.1}, and  follows from  (\ref{seventhree}) by  letting $u=s$, $v=\bar{s} $, $N=[t]$.

Furthermore, the rhs of \eqref{sevensix} can be estimated by using the results of section 2 and in particular Lemmas \ref{l2.1} and \ref{l2.2}. Thus \eqref{sevensix} takes the form
\begin{align}\label{sevensixa}
2\Re&\left\{ 
\sum_{m_{1}=1}^{[t]}\sum_{m_{2}=1}^{[t]}\frac{1}{m_{2}^{\bar{s}}(m_{1}+m_{2})^s } \right\} -
\left( \sum_{m=1}^{[t]}\frac{1}{m^{s}} \right) \left( \sum_{m=1}^{[t]}\frac{1}{m^{\bar{s} }} \right)\nonumber \\ 
&\hspace*{9mm}=\begin{cases}\dfrac{t^{1-2\sigma }}{1-2\sigma }+O \left(  t^{\frac{1}{2}-\frac{5}{3}\sigma} \ln{t} \right)+O(1), \quad &0<\sigma <\frac{1}{2}\\
\ln t+O(1), \quad &\sigma=\frac{1}{2}, \vspace*{4mm} \\
O(1), \quad &\frac{1}{2}<\sigma <1, \end{cases} \quad t\rightarrow \infty.
\end{align}
The details for this estimate are given in Lemma \ref{l7.2}.

\begin{lemma}\label{l7.1} 
Define the functions $f(u,v)$ and $g(u,v)$ by
\begin{equation}  \label{sevenone}
f(u,v)=\sum_{m_{1}=1}^{N}\sum_{m_{2}=1}^{N}\frac{1}{m_{1}^{u}}\frac{1}{(m_{1}+m_{2})^{v}},
\end{equation}
\begin{equation}  \label{seventwo}
g(u,v)=\sum_{m=1}^{N}\sum_{n=N+1}^{N+m}\frac{1}{m^{u}n^{v}},
\end{equation}
where $N$ is an arbitrary finite positive integer and $u\in\mathbb{C}$, $v\in\mathbb{C}$. These functions satisfy the identity 
\begin{equation}  \label{seventhree}
f(u,v)+f(v,u)+\sum_{m=1}^{N}\frac{1}{m^{u+v}}=
\left( \sum_{m=1}^{N}\frac{1}{m^{u}} \right)\left( \sum_{n=1}^{N}\frac{1}{n^{v}} \right) +g(u,v)+g(v,u).
\end{equation}
\end{lemma}

\begin{proof}
Letting $m_{1}=m$, $m_{1}+m_{2}=n$ in $f(u,v)$ and in $f(v,u)$, and then exchanging $m$ and $n$ in the expression of 
$f(v,u)$, we find the following:
\begin{align*}
f(u,v) &+f(v,u)=\left( \sum_{m=1}^{N}\sum_{n=m+1}^{m+N}+\sum_{n=1}^{N}\sum_{m=N+1}^{N+n} \right)\frac{1}{m^{u}n^{v}}  \\
&=\left( \sum_{m=1}^{N}\sum_{n=m+1}^{N}+\sum_{m=1}^{N}\sum_{n=N+1}^{N+m}
+\sum_{n=1}^{N}\sum_{m=n+1}^{N}+\sum_{n=1}^{N}\sum_{m=N+1}^{N+n}    \right)\frac{1}{m^{u}n^{v}}. 
\end{align*}
The second sum above equals $g(u,v)$, and by exchanging $m$ and $n$ in the last sum it follows that the latter sum equals $g(v,u)$. 
Thus, the above identity becomes
\begin{equation} \label{sevenfour}
f(u,v)+f(v,u)=\left( \sum_{m=1}^{N}\sum_{n=m+1}^{N}+\sum_{n=1}^{N}\sum_{m=n+1}^{N} \right)\frac{1}{m^{u}n^{v}}
+g(u,v)+g(v,u).
\end{equation}
But
\begin{equation}  \label{sevenfive}
\sum_{n=1}^{N}\sum_{m=n+1}^{N}\frac{1}{m^{u}n^{v}}=\sum_{m=1}^{N}\sum_{n=1}^{m-1}\frac{1}{m^{u}n^{v}}.
\end{equation}

Using the identity (\ref{sevenfive}) in (\ref{sevenfour}), adding to both sides of (\ref{sevenfour}) the term 
\begin{equation}
\sum_{m=1}^{N}\frac{1}{m^{u}m^{v}},  \nonumber
\end{equation}
and noting that
$$\left( \sum_{m=1}^{N}\sum_{n=m+1}^{N}+\sum_{m=1}^{N}\sum_{n=1}^{m-1}\right)  \frac{1}{m^{u}n^{v}}
+\sum_{m=1}^{N}\frac{1}{m^{u}{m}^{v}}=\sum_{m=1}^{N}\sum_{n=1}^{N}\frac{1}{m^{u}n^{v}}, $$   
equation (\ref{sevenfour}) becomes (\ref{seventhree}).

\end{proof}
%




In order to estimate the rhs of equation \eqref{sevensix}, we use the elementary estimate 
\begin{equation} \label{sevennine}
\sum_{m=1}^{[t]}\frac{1}{m^{2\sigma }}=\begin{cases}\ln t+O(1), \quad &\sigma=\frac{1}{2}, \vspace*{4mm} \\
\dfrac{t^{1-2\sigma }}{1-2\sigma }+O(1), \quad &0<\sigma <1, \quad \sigma \neq\frac{1}{2}, \end{cases} \quad t\rightarrow \infty,
\end{equation}
as well as the result below.


\begin{lemma} \label{l7.2}
The following estimates are valid:
\begin{align} \label{sevenseven}
 2\Re\left\{ \sum_{m=1}^{[t]}\sum_{n=[t]+1}^{[t]+m}\frac{1}{m^{\bar{s}}n^{s} }  \right\}  
= \begin{cases}
O \left(  t^{\frac{1}{2}-\frac{5}{3}\sigma} \ln{t} \right), & 0\leq\sigma\leq \frac{1}{2},\\
O \left(  t^{\frac{1}{3}-\frac{4}{3}\sigma} \ln{t}  \right), & \frac{1}{2} < \sigma< 1,
\end{cases} \qquad t\to \infty.
\end{align}
\end{lemma}

\begin{proof} 

In order to simplify the double sum appearing in the lhs of equation (\ref{sevenseven}) we use relation \eqref{FL-eq}, taking $\eta=2\pi(t+m)$, equivalently $\left[\frac{\eta}{2\pi}\right] =[t]+m$:
\begin{align*}
\sum_{n=[t]+1}^{[t]+m}\frac{1}{n^{s}}&=\frac{1}{1-s}(t+m)^{1-s}+O\left( \frac{1}{t^{\sigma }} \right)  \\
&=i \frac{1}{1+\frac{i(1-\sigma )}{t}}\frac{1}{t^{s}m^{s-1}}\left( \frac{1}{t}+\frac{1}{m} \right)^{1-s} +O\left( \frac{1}{t^{\sigma }} \right) . 
\end{align*}

Replacing in the lhs of \eqref{sevenseven} the sum over $n$ by the above sum we find 
\begin{align}\label{seventhirteen}
2\Re\left\{ \sum_{m=1}^{[t]}\sum_{n=[t]+1}^{[t]+m}\frac{1}{m^{\bar{s} }n^{s}} \right\}=
-2\Im\Bigg\{
\frac{1}{t^{s}}&\sum_{m=1}^{[t]}\frac{1}{m^{2\sigma -1}}\left( \frac{1}{t}+\frac{1}{m}  \right)^{1-s}
\left(1+O\left( \frac{1}{t} \right)\right)   \nonumber\\
&+O\left( \frac{1}{t^{\sigma }} \right)\sum_{m=1}^{[t]}\frac{1}{m^{\bar{s} }}\Bigg\} , \hphantom{3a} t\rightarrow \infty .
\end{align}

The first single sum in the rhs of \eqref{seventhirteen}  involves the function $f(m)$ defined in \eqref{sf1}. Moreover, since $1\leq m\leq t$ and $0<\sigma<1$ we find
$$\frac{1}{m^{2\sigma -1}}\left( \frac{1}{t}+\frac{1}{m}  \right)^{1-\sigma} \leq \frac{1}{m^{2\sigma -1}}\left(\frac{2}{m}  \right)^{1-\sigma}<\frac{2}{m^\sigma}.$$
Thus, the analysis  in the proof of Lemma \ref{l2.2} yields the estimate
$$\sum_{m=1}^{[t]}\frac{1}{m^{2\sigma -1}}\left( \frac{1}{t}+\frac{1}{m}  \right)^{1-s}= \begin{cases}
O \left(  t^{\frac{1}{2}-\frac{2}{3}\sigma} \ln{t} \right), & 0\leq\sigma\leq \frac{1}{2},\\
O \left(  t^{\frac{1}{3}-\frac{1}{3}\sigma} \ln{t}  \right), & \frac{1}{2} < \sigma< 1,
\end{cases} \qquad t\to \infty.$$

For the second single sum in the rhs of \eqref{seventhirteen}  we use the classical estimate \eqref{est0}

Applying the above estimates of the two single sums in \eqref{seventhirteen} yields  \eqref{sevenseven}.

\end{proof}



\section{Further estimates for double exponential sums}

In this section we analyse two of the most well-known types of double exponential sums, namely the  Euler-Zagier and the Mordell-Tornheim sums. In this section we do not restrict the real parts of the exponents in the interval $(0,1)$.

\subsection{Special cases of Euler-Zagier with different exponents}

\begin{lemma}\label{lemma4.1}
Let $S_A$ denote double sum
\begin{equation}\label{sixTsup0}
S_A=\sum_{m_1=1}^{[t]}\sum_{m_2=1}^{[t]} \dfrac{1}{\left(m_1+m_2\right)^{\sigma_1+i t}}\dfrac{1}{m_2^{\sigma_2-i t}},
\end{equation}
with $\sigma_1<0$ and $\sigma_2>1$. Then,
\begin{equation}\label{sixTsup3}
\big|S_A\big|=O\left(t^{\frac{1}{2}-\sigma_1} \ln t \right), \qquad  t\to \infty.
\end{equation}
\end{lemma}
\begin{proof}
Letting $m_2=m, \ m_1+m_2=n,$ and employing the triangular inequality we find
\begin{equation}\label{sixTsup1}
\big|S_A\big|=\Bigg|\sum_{m=1}^{[t]}\sum_{n=m+1}^{m+[t]} \dfrac{1}{n^{\sigma_1+i t}}\dfrac{1}{m^{\sigma_2-i t}}\Bigg| \leq \sum_{m=1}^{[t]} \Bigg| \sum_{n=m+1}^{m+[t]} \dfrac{1}{n^{\sigma_1+i t}}\Bigg| \dfrac{1}{m^{\sigma_2}}.
\end{equation}
Taking into consideration \eqref{est0} with $\sigma_1=0$, we find
\begin{equation*}\label{sixTsup2a}
\sum_{n=1}^{[t]} \dfrac{1}{n^{i t}} = O\left(t^{\frac{1}{2}} \ln t\right), \qquad  t\to \infty.
\end{equation*}
Applying partial summation we obtain the estimate
\begin{equation}\label{sixTsup2b}
\sum_{n=m+1}^{m+[t]} \dfrac{1}{n^{\sigma_1+i t}}=O\left(t^{\frac{1}{2}-\sigma_1} \ln t\right), \qquad  t\to \infty,
\end{equation}
for $\sigma_1<0$ and $1\leq m \leq [t].$

Indeed, using \eqref{sixTsup2b} into \eqref{sixTsup1} and noting that $\sigma_2>1$, we find \eqref{sixTsup3}

\end{proof}

\begin{remark}
An alternative proof of \eqref{sixTsup2b} can be derived by using the estimate 
\begin{equation}\label{sixrem7}
\sum_{m=1}^{[t]} \dfrac{1}{m^{\sigma-1+it}} = O\left(t^{\frac{3}{2}-\sigma}\right), \qquad 0<\sigma<1,\quad t\to\infty,
\end{equation}
for $\sigma_1=\sigma-1<0$.

The proof of \eqref{sixrem7} is provided in the Appendix A.

\end{remark}

%
%

\subsection{Special cases of Mordell-Tornheim sums}

\begin{lemma}\label{lemma4.2}
Let $S_B$ denote  double sum
\begin{equation}\label{sixTsup4}
S_B=\sum_{m_1=1}^{[t]}\sum_{m_2=1}^{[t]} \dfrac{1}{ \left( m_1+m_2 \right)^{\sigma_1+i t}} \dfrac{1}{m_2^{\sigma_2-i t}} \dfrac{1}{m_1^{\sigma_3}},
\end{equation}
with $\sigma_1<0, \ \sigma_2\in (0,1)$ and $\sigma_3\geq 1$.
Then, 
\begin{equation}\label{sixTsup15}
\big|S_B\big| =  \begin{cases} O \left(t^{1-\sigma_1-\sigma_2}\ln t\right),  & 0 < \sigma_2 <\frac{1}{2}, \quad \sigma_3 =1  , \\ O \left(t^{1-\sigma_1-\sigma_2}\right),  & 0 < \sigma_2 <\frac{1}{2}, \quad \sigma_3 >1  ,\\ O\left(t^{\frac{1}{2}-\sigma_1}\ln t\right), & \frac{1}{2} \leq  \sigma_2 <1, \quad \sigma_3 \geq 1, \end{cases} \quad t \to \infty.
\end{equation}
\end{lemma}
\begin{proof}
Splitting this sum into two sums, depending on whether $m_1/m_2>1$ or $m_1/m_2<1,$ we find
\begin{equation}\label{sixTsup5}
S_B=S_1+S_2,
\end{equation}
where
\begin{equation}\label{sixTsup6}
S_1=\sum_{m_1=1}^{[t]}\sum_{m_2=1}^{m_1} \dfrac{1}{\left(m_1+m_2\right)^{\sigma_1+i t}}\dfrac{1}{m_2^{\sigma_2-i t}}\dfrac{1}{m_1^{\sigma_3}},
\end{equation}
and
\begin{equation}\label{sixTsup7}
S_2=\sum_{m_1=1}^{[t]}\sum_{m_2=m_1+1}^{[t]} \dfrac{1}{\left(m_1+m_2\right)^{\sigma_1+i t}}\dfrac{1}{m_2^{\sigma_2-i t}}\dfrac{1}{m_1^{\sigma_3}}.
\end{equation}
In order to estimate the sum $S_1$, we change the order of summation, see figure \ref{Fig-s}.
\begin{figure}
\begin{center}
\includegraphics[scale=0.25]{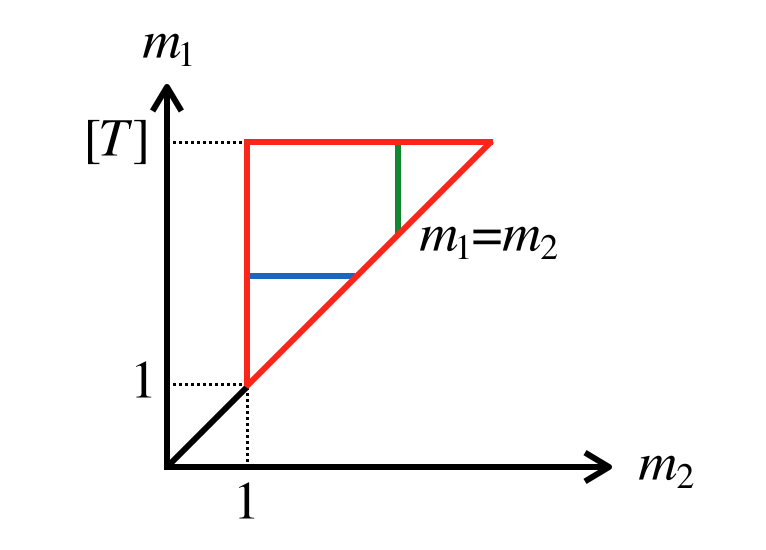}
\end{center}
\caption{Change of the order of summation.}
\label{Fig-s}
\end{figure}

Thus,
\begin{equation*}
S_1=\sum_{m_2=1}^{[t]}\sum_{m_1=m_2}^{[t]} \dfrac{1}{\left(m_1+m_2\right)^{\sigma_1+i t}}\dfrac{1}{m_2^{\sigma_2-i t}}\dfrac{1}{m_1^{\sigma_3}},
\end{equation*}
or
\begin{equation}\label{sixTsup8}
S_1=\sum_{m_2=1}^{[t]}\sum_{m_1=m_2}^{[t]} \dfrac{1}{\left(m_1+m_2\right)^{\sigma_1+i t}}\dfrac{1}{m_2^{\sigma_2+1-i t}}\dfrac{m_2}{m_1^{\sigma_3}}.
\end{equation}
Using partial summation and the fact that $\frac{m_2}{m_1^{\sigma_3}}\leq 1,$ it follows that
\begin{equation}\label{sixTsup9}
S_1=O\left(| \tilde{S}_1 | \right), \qquad  t\to \infty,
\end{equation}
where
\begin{equation}\label{sixTsup10}
\tilde{S}_1=\sum_{m_2=1}^{[t]}\sum_{m_1=m_2}^{[t]} \dfrac{1}{\left(m_1+m_2\right)^{\sigma_1+i t}}\dfrac{1}{m_2^{\sigma_2+1-i t}}.
\end{equation}
Then, proceeding as with the sum $S_A$ in \eqref{sixTsup1}, we obtain the estimate \eqref{sixTsup3}, i.e.,
\begin{equation}\label{sixTsup11}
S_1=O\left(t^{\frac{1}{2}-\sigma_1}\ln t\right), \qquad  t\to \infty.
\end{equation}
In order to estimate $S_2$, we first note that
\begin{equation}\label{sixTsup12}
\big|S_2\big|\leq \sum_{m_1=1}^{[t]} \Bigg| \sum_{m_2=m_1+1}^{[t]} \dfrac{1}{\left(m_1+m_2\right)^{\sigma_1+i t}}\dfrac{1}{m_2^{\sigma_2-i t}}\Bigg| \dfrac{1}{m_1^{\sigma_3}}.
\end{equation}
Then, taking into consideration that $m_1<m_2$, we can use the following ``crude" estimate for the $m_2$ sum:
\begin{equation}\label{sixTsup13}
\Bigg| \sum_{m_2=m_1+1}^{[t]} \dfrac{1}{\left(m_1+m_2\right)^{\sigma_1+i t}}\dfrac{1}{m_2^{\sigma_2-i t}}\Bigg| \leq \int_{m_1+1}^t \dfrac{1}{\left(m_1+x\right)^{\sigma_1}}\dfrac{1}{x^{\sigma_2}} dx:=J\left(m_1,t\right).
\end{equation}
But, $$m_1<x, \ \text{ or } \quad m_1+x<2x,  \ \text{ or } \quad \left(m_1+x\right)^{-\sigma_1} <(2x)^{-\sigma_1}.$$
Thus, $$J\left(m_1,t\right)< \int_{m_1+1}^t 2^{-\sigma_1} x^{-\sigma_1-\sigma_2} dx= O \left(t^{1-\sigma_1-\sigma_2}\right) + O \left(m_1^{1-\sigma_1-\sigma_2}\right)= O \left(t^{1-\sigma_1-\sigma_2}\right), \quad  t\to \infty,$$
since $\sigma_1+\sigma_2<1$.
Hence, equations \eqref{sixTsup12} and \eqref{sixTsup13} yield
\begin{equation}\label{sixTsup14}
\big|S_2\big| = O \left(t^{1-\sigma_1-\sigma_2}\int_{1}^t \frac{dx}{x^{\sigma_3}}\right)= O \left(t^{1-\sigma_1-\sigma_2}\right) \times  \begin{cases} O(\ln t) , & \sigma_3 =1,  \\ O (1), & \sigma_3 >1,  \end{cases} \qquad  t\to \infty.
\end{equation}
\end{proof}

\begin{remark}
One can apply the estimate used in \eqref{sixTsup14} to  $S_1$, and then the estimates \eqref{sixTsup11} and \eqref{sixTsup15} should be substituted by \eqref{sixTsup14}. Furthermore, for the special cases $\sigma_1=\sigma-1, \ \sigma_2=\sigma$ and $\sigma_3=1$, with $\sigma\in (0,1)$,  the estimates \eqref{sixTsup3} and \eqref{sixTsup15} take the form
\begin{equation}\label{rem4.2-1}
\big|S_A\big|=O\left(t^{\frac{3}{2}-\sigma} \ln t \right), \qquad  t\to \infty,
\end{equation}
and
\begin{equation}\label{rem4.2-2}
\big|S_B\big| =  \begin{cases} O \left(t^{2-2\sigma}\ln t\right),  & 0 < \sigma <\frac{1}{2},
\\ O\left(t^{\frac{3}{2}-\sigma}\ln t\right), & \frac{1}{2} \leq  \sigma <1, \end{cases} \qquad  t\to \infty,
\end{equation}
respectively.
\end{remark}

\section{Double sums for ``small'' sets of summation}

The analysis presented in \cite{Fs} requires estimating the following sum:
\begin{equation}\label{seven_n_1}
\mathop{\sum\sum}_{(m_1,m_2) \in M}\frac{1}{m_{1}^{s}(m_{1}+m_{2})^{\bar{s} }},
\end{equation}
where $M$ is defined by 
\begin{align}  \label{sevenforteen}
M = \Bigg\{ m_{1}\in\mathbb{N}^{+}, \   m_{2}\in  \mathbb{N}^{+}, & \quad 1\leq m_{1}\leq [t], \quad 1\leq m_{2}\leq [t],  \notag \\
& \frac{1}{t^{1-\delta _{2}}-1}<\frac{m_{2}}{m_{1}}<t^{1-\delta _{3}}-1, \quad  t>0\Bigg\} ,
\end{align}
with $\delta _{2}$ and $\delta _{3}$ positive constants.

The above sum can be related to the sum appearing in the first term of the lhs of \eqref{sevensix} via the following identity:
\begin{equation}\label{sevenfifteen}
\sum_{m_{1}=1}^{[t]}\sum_{m_{2}=1}^{[t]}\frac{1}{m_{1}^{s}(m_{1}+m_{2})^{\bar{s} }}=\mathop{\sum\sum}_{(m_{1},m_{2})\in M}\frac{1}{m_{1}^{s}(m_{1}+m_{2})^{\bar{s} }}+ S_1(\sigma,t,\delta_3) + S_2(\sigma,t,\delta_2),
\end{equation}
 with
\begin{equation}\label{sevenfifteen-b}
S_1(\sigma,t,\delta_3) =
\sum_{m_{1}=1}^{\left[\frac{t}{t^{1-\delta _{3}}-1}\right]-1}
\sum_{m_{2}=\left[(t^{1-\delta _{3}}-1)m_{1}\right]+1}^{[t]} \frac{1}{m_{1}^{s}(m_{1}+m_{2})^{\bar{s} }}
\end{equation}
and
\begin{equation}\label{sevenfifteen-c}
S_2(\sigma,t,\delta_2)=\sum_{m_{1}=\left[t^{1-\delta _{2}}\right]}^{[t]}\sum_{m_{2}=1}^{\left[\frac{m_{1}}{t^{1-\delta _{2}}-1}\right]-1}\frac{1}{m_{1}^{s}(m_{1}+m_{2})^{\bar{s} }}.
\end{equation}

Thus, estimating the sum \eqref{seven_n_1} requires estimating the sum $S_1$ and $S_2$. The relevant estimates are presented in Theorems \ref{lemma7.4} and \ref{l7.5} below.

By making the change of variables $m_1=m$ and $m_1+m_2=n$ in \eqref{sevenfifteen-b} we can rewrite $S_1$ in the form
\begin{equation}\label{sevenfifteen-d}
S_1(\sigma,t,\delta_3) =
\sum_{m=1}^{\left[\frac{t}{t^{1-\delta _{3}}-1}\right]-1}
\sum_{n=\left[t^{1-\delta_3 }m\right]+1}^{[t]+m}
\frac{1}{m^{s}n^{\bar{s}}}.
\end{equation}

Using the equation
 $$\frac{t}{t^{1-\delta _{3}}-1}=\frac{t}{t^{1-\delta _{3}}\left(1-t^{\delta_3-1}\right)}=t^{\delta _{3}}\left( 1 + O\left(t^{\delta_3-1}\right)\right), \ \ 0<\delta_3<1, \qquad t\to \infty,$$ 
it follows that for $\delta_3<1/2$, the upper bound of the expression $\left[\frac{t}{t^{1-\delta _{3}}-1}\right]-1$  is equal either to $\left[t^{\delta _{3}}\right]$  or  to $\left[t^{\delta _{3}}\right]-1$. Thus, it is sufficient to consider the following form of $S_1$:
\begin{equation}\label{sevenfifteen-e}
S_1(\sigma,t,\delta_3) =
\sum_{m=1}^{\left[t^{\delta _{3}}\right]}
\sum_{n=\left[t^{1-\delta_3 }m\right]+1}^{[t]+m}
\frac{1}{m^{s}n^{\bar{s}}}.
\end{equation}

Regarding the sum $S_2$,  by using the fact that $$\frac{m_1}{t^{1-\delta _{2}}-1}=\frac{m_1}{t^{1-\delta _{2}}}\left(1 + O\left(t^{\delta_2-1}\right)\right)=\frac{m_1}{t^{1-\delta _{2}}} + O\left(t^{2\delta_2-1}\right), \ \ 0<\delta_2<1, \qquad t\to \infty,$$
we conclude that $\left[\frac{m_1}{t^{1-\delta _{2}}-1}\right]-1$ is equal either to $\left[\frac{m_{1}}{t^{1-\delta _{2}}}\right]-1$ or to $\left[\frac{m_{1}}{t^{1-\delta _{2}}}\right]$, for $\delta_2<1/2$.

Thus, it is sufficient to consider the following form of $S_2$:
\begin{equation}\label{sevenfifteen-f}
S_2(\sigma,t,\delta_2)=\sum_{m_{1}=\left[t^{1-\delta _{2}}\right]}^{[t]}\sum_{m_{2}=1}^{\left[\frac{m_{1}}{t^{1-\delta _{2}}}\right]}\frac{1}{m_{1}^{s}(m_{1}+m_{2})^{\bar{s} }}.
\end{equation}

\begin{theorem} \label{lemma7.4}
Define the double sum $S_{1}$ by 
\begin{equation} \label{seveneighteen}
S_{1}(\sigma,t,\delta )=\sum_{m=1}^{\left[t^\delta\right]}\sum_{n=\left[t^{1-\delta }m\right]+1}^{[t]+m}
\frac{1}{m^{s}n^{\bar{s}}}, \ \qquad 0<\delta<1, \quad
s=\sigma +it, \quad 0<\sigma <1, \quad t>0.
\end{equation}
 Then,
\begin{equation} \label{sevennineteen}
S_{1}(\sigma,t,\delta ) = O\left( t^{\frac{1}{2}-\sigma}\tilde{G}(\sigma,t,\delta) \right) + O\left( \frac{t^{(1-\sigma)\delta}}{t^\sigma} \right), \qquad 0<\sigma<1, ~ t\to\infty,
\end{equation}
where
\begin{equation} \label{seventwenty}
\tilde{G}(\sigma,t,\delta ) = O \left(  t^{(1-\sigma)\delta} \right) + O \left(  t^{\sigma\delta} \right), \qquad 0<\sigma<1, \quad \sigma\ne\frac{1}{2}, \quad t\to\infty,
\end{equation}
and
\begin{equation} \label{seventwentyone}
\tilde{G}\left(\frac{1}{2},t,\delta\right) = O \left(  t^{\frac{\delta}{2}} \ln{t} \right), \qquad t\to\infty.
\end{equation}
\end{theorem}

\begin{proof}

It is convenient to split the $S_1$ sum in terms of the following two sums:
\begin{equation} \label{seventwentyseven}
S_{A}(\sigma,t,\delta )=\sum_{m=1}^{\left[t^{\delta}\right]}\sum_{n=\left[ t^{1-\delta }m\right]+1}^{[t]}
\frac{1}{m^{s}n^{\bar{s}}}, \quad 0<\sigma <1, \quad t>0,
\end{equation}
and
\begin{equation} \label{seventwentyeight}
S_{B}(\sigma,t,\delta )=\sum_{m=1}^{\left[t^{\delta}\right]}\sum_{n=[t]+1}^{[t]+m}
\frac{1}{m^{s}n^{\bar{s}}}, \quad 0<\sigma <1, \quad t>0.
\end{equation}
Thus, computing $S_1$ reduces to computing $S_A$ and $S_B$:
\begin{equation} \label{seventwentynine}
S_{1}(\sigma,t,\delta ) = S_{A}(\sigma,t,\delta ) + S_{B}(\sigma,t,\delta ).
\end{equation}
We first analyze $S_B$. In order to estimate the $n$-sum of $S_B$ we employ the identity \eqref{FL-eq} with $\eta=2\pi (t +m)$, equivalently $\left[\frac{\eta}{2\pi}\right]=[t]+m$ :
\begin{equation}  \label{seventhirty}
\sum_{n=[t]+1}^{[t]+m}\frac{1}{n^{\bar{s}}}=\frac{1}{1-\bar{s}}\left(t+m\right)^{1-\bar{s}}+O(t^{-\sigma}), \quad 0<\sigma <1, \quad t\to\infty.
\end{equation}
We note that
\begin{align*}
\frac{1}{1-\bar{s}} \frac{\left(t+m\right)^{1-\bar{s}}}{m^s} &= \frac{1}{1-\sigma + it} \frac{\left(t+m\right)^{1-\sigma}}{m^{\sigma}} \frac{\left(t+m\right)^{it}}{m^{it}} \\ 
&= - \frac{i}{1-\frac{i(1-\sigma)}{t}} t^{-\bar{s}} \left( 1+\frac{m}{t} \right)^{1-\sigma} \frac{1}{m^{\sigma}} \left( \frac{1}{t} + \frac{1}{m}  \right)^{it}.
\end{align*}
Using this expression in \eqref{seventhirty} and then substituting the resulting sum in \eqref{seventwentyeight} we find
\begin{multline} \label{seventhirtyone}
S_{B}(\sigma,t,\delta )= O(t^{-\sigma}) \sum_{m=1}^{\left[ t^{\delta}\right]}\frac{1}{m^{s}} + O(t^{-\sigma}) \sum_{m=1}^{\left[ t^{\delta}\right]} \left\{ \left( 1 + \frac{m}{t}\right)^{1-\sigma} \right.\\
\left. \times  \frac{1}{m^{\sigma}} \left( \frac{1}{t} + \frac{1}{m}\right)^{it}\right\}, \quad 0<\sigma <1, \quad t\to\infty.
\end{multline}

Using the fact that the function
\begin{equation*}
\left( 1 + \frac{m}{t}\right)^{1-\sigma}, \quad 1\le m \le t^{\delta}, \quad 0\le\sigma\le 1, \quad t>0,
\end{equation*}
is bounded, and employing the classical result on partial summation of single sums, see for example 5.2.1 of \cite{T}, it is possible to associate the second sum appearing in \eqref{seventhirtyone} with 
\begin{equation} \label{seventhirtytwo}
\tilde{S}_{B}(\sigma,t,\delta) = \sum_{m=1}^{\left[t^{\delta}\right]} \frac{1}{m^{\sigma}}e^{if(m)},
\end{equation}
where $f(m)$ is defined in \eqref{sf1}. Furthermore, recalling that  $\tilde{S}_{B}$ can be estimated using \eqref{B.7}, we obtain
\begin{equation} \label{seventhirtythree}
\tilde{S}_{B}(\sigma,t,\delta) = O\left( t^{(1-\sigma)\delta} \right), \quad t\to\infty,
\end{equation}
hence, it follows that
\begin{equation*}
\left|\sum_{m=1}^{\left[t^{\delta}\right]} \left(1+ \frac{m}{t}\right)^{1-\sigma}\frac{1}{m^{\sigma}} \left( \frac{1}{t} + \frac{1}{m}\right)^{it} \right| = O\left( t^{(1-\sigma)\delta} \right), \qquad  t\to \infty.
\end{equation*}

The first sum in \eqref{seventhirtyone} satisfies an identical estimate with the above, and then equation \eqref{seventhirtyone} implies
\begin{equation} \label{seventhirtyfour}
S_{B}(\sigma,t,\delta) = O\left( t^{-\sigma+(1-\sigma)\delta} \right), \quad t\to\infty.
\end{equation}

We next analyze $S_A$.  For the evaluation of the $n$-sum in the double sum $S_A$ defined in \eqref{seventwentyseven}  we will employ the  asymptotic formula \eqref{FR} with
\begin{equation*}
\frac{t}{\eta_1} = t, \quad \text{i.e.}, \quad \eta_1 = 1,
\end{equation*}
and
\begin{equation*}
\frac{t}{\eta_2} + 1 = t^{1-\delta}m +1, \quad \text{i.e.}, \quad \eta_2 = \frac{t^{\delta}}{m}.
\end{equation*}
If $m=1$ then $\eta_2=t^{\delta}$, and if $m=t^{\delta}$ then $\eta_2=1$. Thus, the inequalities in \eqref{FR} are satisfied and hence equation \eqref{FR} yields
\begin{equation} \label{seventhirtyeight}
\sum_{n=\left[t^{1-\delta}m\right]+1 }^{[t]} \frac{1}{n^{\bar{s}}} = \chi(\bar{s}) \sum_{n=1}^{\left[ \frac{t^{\delta}}{2\pi m}\right] }  \frac{1}{n^{1-\bar{s}}} + \bar{E}\left(\sigma,t,\frac{t^{\delta}}{m}\right) -  \bar{E}(\sigma,t,1),~ 0<\sigma<1, ~ t\to\infty.
\end{equation}

Inserting \eqref{seventhirtyeight} into the definition \eqref{seventwentyseven} of $S_A$ we find
\begin{multline} \label{seventhirtynine}
S_A(\sigma,t,\delta) = \chi(\bar{s}) \sum_{m=1}^{\left[ t^{\delta} \right] } \sum_{n=1}^{ \left[ \frac{t^{\delta}}{2\pi m}\right] } \frac{1}{m^s} \frac{1}{n^{1-\bar{s}}} + \chi(\bar{s})  \sum_{m=1}^{\left[t^{\delta}\right]} \left[  \bar{E}\left(\sigma,t,\frac{t^{\delta}}{m}\right) -  \bar{E}(\sigma,t,1) \right], \\
\quad 0<\sigma<1, \quad t\to\infty.
\end{multline}
The occurrence of the term $t^{\delta}$ in the above sums implies that these sums can be easily estimated:
\begin{multline} \label{sevenforty}
\left| \sum_{m=1}^{\left[ t^{\delta}\right]} \sum_{n=1}^{ \left[ \frac{t^{\delta}}{2\pi m} \right] } \frac{1}{m^s} \frac{1}{n^{1-\bar{s}}} \right| \le \int_{1}^{t^{\delta}} \text{d}x ~ x^{-\sigma} \int_{1}^{\frac{t^{\delta}}{2\pi x}} \text{d}y ~ y^{\sigma-1} = \tilde{G}(\sigma,t,\delta), \\
0<\sigma<1, \quad t\to\infty,
\end{multline}where
\begin{equation*}
\tilde{G}(\sigma,t,\delta) = O \left(  t^{(1-\sigma)\delta} \right) + O \left(  t^{\sigma\delta} \right), \quad 0<\sigma<1, \quad \sigma \ne \frac{1}{2}, \quad t\to\infty,
\end{equation*}
and
\begin{equation*}
\tilde{G}\left( \frac{1}{2} ,t,\delta \right) = O \left(  t^{\frac{\delta}{2}} \ln{t} \right), \quad t\to\infty.
\end{equation*}

Recalling the asymptotic formula
\begin{equation} \label{sevenfortyone}
\chi(s) = \left(\frac{2\pi}{t}\right)^{s-\frac{1}{2}} e^{it}  e^{\frac{i\pi}{4}} \left( 1 +O \left( \frac{1}{t}\right) \right), \quad 0<\sigma<1, \quad t\to\infty,
\end{equation}
which is derived in the Appendix A of \cite{FL}, it follows that
\begin{equation} \label{sevenfortytwo}
S_A = O \left( t^{\frac{1}{2}-\sigma}\right) \tilde{G}(\sigma,t, \delta), \quad t\to\infty.
\end{equation}
Equations \eqref{seventwentynine}, \eqref{seventhirtyfour} and \eqref{sevenfortytwo} imply \eqref{sevennineteen}.

\end{proof}

For the estimation of $S_A$, which gives the dominant contribution of $S_1$, one can also use an alternative approach, which is based on classical techniques appearing in \cite{T,T2}, and obtain slightly weaker, but essentially similar results. In this connection we obtain the following Lemma:

\begin{lemma}\label{lemma5.1}
Let $S_A$ be defined by \eqref{seventwentyseven}. Then
\begin{equation}\label{est-n-cl-1}
S_A=O\left( t^{\frac{1}{2}-\sigma} \ln t \right) \tilde{G}(\sigma,t,\delta) , \qquad 0<\sigma<1, ~ t\to\infty.
\end{equation}
\end{lemma}
\begin{proof}
Observing that $m$ takes relatively ``small" values in the set of summation of $S_A$, we use the following inequality without losing crucial information $$\left| S_A\right| <\sum_{m=1}^{\left[t^{\delta}\right]} \frac{1}{m^\sigma}\left|\sum_{n=\left[ t^{1-\delta }m\right]+1}^{[t]}
\frac{1}{n^{\bar{s}}} \right| . $$
Then, we estimate the $n$-sum using Theorem 5.9 of \cite{T}, namely
$$\sum_{a<n\leq b\leq 2a} n^{it} = O\left(t^\frac{1}{2}\right) + O\left(a t^{-\frac{1}{2}}\right).$$
Using partial summation and the fact that $a>m  t^{1-\delta }$, similarly to the proof of Theorem 5.12 of \cite{T}, we obtain that 
$$\sum_{n=\left[ t^{1-\delta }m\right]+1}^{[t]} \frac{1}{n^{\bar{s}}} = O\left( t^\frac{1}{2} t^{-(1-\delta)\sigma} m^{-\sigma} \ln t\right), \qquad t\to\infty.$$

Thus,
\begin{equation}\label{est-n-cl-2}
S_A=\sum_{m=1}^{\left[t^{\delta}\right]} \frac{1}{m^{2\sigma}} O\left( t^{\frac{1}{2}-\sigma} t^{\delta\sigma} \ln t\right), \qquad t\to\infty.
\end{equation}

Applying in \eqref{est-n-cl-2} the fact that
$$\sum_{m=1}^{\left[t^{\delta}\right]} \frac{1}{m^{2\sigma}}=\begin{cases} O\left(t^{(1-2\sigma)\delta}\right), & 0<\sigma<\frac{1}{2},\\
O(\ln t),  & \sigma=\frac{1}{2}, \\
O(1), & \frac{1}{2}<\sigma<1,
\end{cases}$$
yields \eqref{est-n-cl-1}.
\end{proof}

\begin{remark}
The estimates of $S_A$ given in \eqref{sevenfortytwo} and  \eqref{est-n-cl-1} differ only by a $\ln t$ term. The approach in the proof of Lemma \ref{lemma5.1} implies that for $0<\delta<1/3,$ the estimate of $S_A$ is essentially the best which one should expect via the classical techniques presented in \cite{T,T2,K}. In particular, for $\sigma=1/2$, these techniques together with Theorem 5.14 of \cite{T}, suggest the estimate
\begin{equation*}
S_A\left(\frac{1}{2},t,\delta\right)=\begin{cases}O\left( t^{\frac{\delta}{2}} (\ln t)^2 \right) , & 0<\delta<\frac{1}{3},\\ 
O\left( t^{\frac{1}{6}} (\ln t)^2 \right) , & \frac{1}{3}\leq\delta<1, \end{cases}
\qquad t\to\infty.
\end{equation*}
Theorem 1 of \cite{T2} together with the Theorem 2.16 of \cite{K}, does not appear to give an essential improvement of the above estimate.
\end{remark}

\begin{theorem} \label{l7.5}
Define the double sum $S_2(\sigma,t,\delta)$ by
\begin{equation} \label{sevenfortythree}
S_{2}(\sigma,t,\delta )=\sum_{m_{1}=\left[ t^{1-\delta} \right]}^{[t]} \sum_{m_{2}=1}^{\left[\frac{m_1}{t^{1-\delta}}\right]}
\frac{1}{m_{1}^{s}(m_{1}+m_{2})^{\bar{s} }}, \quad 0<\delta<1, \ s=\sigma+it, ~ 0<\sigma<1, ~t>0.
\end{equation}
 Then,
\begin{equation} \label{sevenfortyfour}
S_{2}(\sigma,t,\delta ) = O \left(  t^{1-2\sigma+(2\sigma+1)\delta} \right), \quad 0<\sigma<1, \quad t\to\infty.
\end{equation}
\end{theorem}

\begin{proof}
We find more convenient to treat this sum using some of the `crude' methods, involving the integration, in order to benefit from the smallness of the set of summation.
Indeed, we observe that
\begin{align}\label{est-cr-n-1}
\left| S_2 \right| \leq \int_{t^{1-\delta}}^{t} \int_{1}^{\frac{x}{t^{1-\delta}}}
\frac{1}{x^{\sigma}(x+y)^{\sigma }} dy dx := J_2(t).
\end{align}
Using the fact that $t>x>t^{1-\delta}$, as well as that $x+y>t^{1-\delta}$, then
\begin{align*}
J_2(t)&<\frac{1}{t^{2\sigma(1-\delta)}} \int_{t^{1-\delta}}^{t} \int_{1}^{\frac{x}{t^{1-\delta}}} dy dx<\frac{1}{t^{2\sigma(1-\delta)}} \int_{t^{1-\delta}}^{t} \int_{1}^{t^\delta} dy dx\\
&=\frac{1}{t^{2\sigma(1-\delta)}} \left( t-t^{1-\delta}\right)\left(t^\delta -1\right) =O \left(  t^{1-2\sigma+(2\sigma+1)\delta} \right).
\end{align*}

\end{proof}

\begin{remark}
Using the techniques developed in \cite{T2} and \cite{K} as are appropriately modified  in Appendix B, we obtain the slightly better estimate
\begin{equation} \label{sevenfiftyeight}
S_2(\sigma,t,\delta) = O \left( t^{1-2\sigma} t^{2\delta\sigma} (\ln{t})^3 \right), \quad 0<\sigma<1, \quad t \to\infty.
\end{equation}

The fact that this result does not provide a significant improvement to \eqref{sevenfortyfour} is due to the fact that in the latter approach we have exploited the smallness of the set of summation via the integration process.
\end{remark}


It is possible to improve further the estimate \eqref{sevenfortyfour}, by obtaining a more accurate estimate of the integral $J_2(t)$. Indeed, applying the following Lemma to \eqref{est-cr-n-1}, we obtain 
\begin{equation}\label{est-cr-n-3}
S_2=O\left(t^{1-2\sigma+\delta}\right),  \qquad 0<\sigma<1, \quad t\to\infty.
\end{equation}

\begin{lemma}\label{l5.2}
Let $J_2(t)$ be defined by \eqref{est-cr-n-1}. Then,
\begin{equation} \label{est-cr-n-2}
J_{2}(t) = \dfrac{t^{1-2\sigma+\delta}}{2(1-\sigma)}  \left(1 + O \left(t^{-2\delta(1-\sigma)}, t^{-\delta} \right) \right),
 \quad 0<\sigma<1, \quad t\to\infty.
\end{equation}
\end{lemma}
\begin{proof}
\begin{align*}
J_2(t)&=
\int_{t^{1-\delta}}^{t} \int_{1}^{\frac{x}{t^{1-\delta}}} \frac{1}{x^{\sigma}(x+y)^{\sigma }} dy dx =
\int_{t^{1-\delta}}^{t}  \frac{1}{x^{\sigma}} \left[ x^{1-\sigma}\frac{\left(1+t^{\delta-1}\right)^{1-\sigma}}{1-\sigma} - \frac{(x+1)^{1-\sigma}}{1-\sigma}\right]dx\\
&=\frac{1}{1-\sigma}\int_{t^{1-\delta}}^{t} x^{1-2\sigma} \left[ \left(1+t^{\delta-1}\right)^{1-\sigma} - \left(1+\frac{1}{x}\right)^{1-\sigma}\right]dx\\
&=\frac{1}{1-\sigma}\int_{t^{1-\delta}}^{t} x^{1-2\sigma} \left[ 1+(1-\sigma)t^{\delta-1} - 1 - \frac{1-\sigma}{x} +O\left(t^{2(\delta-1)},\frac{1}{x^2}\right)\right]dx, \quad t\to\infty.
\end{align*}
Using the fact that $x>t^{1-\delta}$, the above integral takes the form
\begin{align*}
J_2(t)&= t^{\delta-1}\int_{t^{1-\delta}}^{t} x^{1-2\sigma}dx - \int_{t^{1-\delta}}^{t} x^{-2\sigma} dx + \int_{t^{1-\delta}}^{t} x^{1-2\sigma}dx \  O\left(t^{2(\delta-1)}\right), \quad t\to\infty,
\end{align*}
which yields \eqref{est-cr-n-2}.
\end{proof}

\begin{theorem} \label{thm5.3}
Let $\delta\in\left(0,\frac{1}{2}\right)$ and the double sum $S_2(\sigma,t,\delta)$ be defined by \eqref{sevenfortythree}.
 Then,
\begin{equation} \label{est-s2-1}
S_{2}(\sigma,t,\delta ) = \begin{cases} O \left(  t^{1-2\sigma} \right) + O \left(  t^{\delta- 2\sigma}\right), \quad & 0<\sigma<\frac{1}{2},\\
 O \left(  \ln t \right) + O \left(  t^{\delta-1} \right), \quad & \sigma=\frac{1}{2}, \end{cases}
 \quad t\to\infty.
\end{equation}
\end{theorem}

\begin{proof}

Letting $m_1=m$ and $m_1+m_2=n$ in the definition \eqref{sevenfortythree} of $S_2$ we find
\begin{equation} \label{sevenfortyfive}
S_{2}(\sigma,t,\delta ) = \sum_{m=\left[ t^{1-\delta}\right]}^{[t]} \sum_{n=1+m}^{\left[m\left(1+t^{\delta-1}\right)\right]}
\frac{1}{m^{s}n^{\bar{s} }}.
\end{equation}
It is convenient to split the $S_2$ sum in terms of the following two sums:
\begin{equation} \label{sevenfortysix}
S_{A}(\sigma,t,\delta )=\sum_{m=\left[ t^{1-\delta}\right]}^{[t]}\sum_{n=1+m}^{P(t)}
\frac{1}{m^{s}n^{\bar{s}}}, \quad 0<\sigma <1, \quad t>0,
\end{equation}
and
\begin{equation} \label{sevenfortyseven}
S_{B}(\sigma,t,\delta )=\sum_{m=\left[ t^{1-\delta}\right]}^{[t]}\sum_{n=[t]+1}^{\left[ m\left(1+t^{\delta-1}\right)\right]}
\frac{1}{m^{s}n^{\bar{s}}}, \quad 0<\sigma <1, \quad t>0,
\end{equation}
where $P(t)=\min\left\{[t], \left[ m\left(1+t^{\delta-1}\right)\right]\right\}$.

Hence
\begin{equation} \label{sevenfortyeight}
S_{2}(\sigma,t,\delta) = S_{A}(\sigma,t,\delta) + S_{B}(\sigma,t,\delta), \quad 0<\sigma <1, \quad t>0.
\end{equation}
We first analyze $S_B$. In this connection we define the function $l(t)$ by
\begin{equation} \label{sevenfortynine}
l(t) =\left[ t-t^{\delta}\right]+1, \quad t>0.
\end{equation}
We observe that the upper limit of the $n$-sum of $S_B$ is greater or equal to $[t]+1$  only if $m\ge l(t)$. Thus, we rewrite $S_B$ in the form
\begin{equation} \label{sevenfifty}
S_{B}(\sigma,t,\delta )=\sum_{m=l(t)}^{[t]}\sum_{n=[t]+1}^{\left[ m\left(1+t^{\delta-1}\right)\right]}
\frac{1}{m^{s}n^{\bar{s}}}, \quad 0<\sigma <1, \quad t>0.
\end{equation}
In order to estimate the $n$-sum of $S_B$ we employ the identity \eqref{FL-eq} with $\eta=2\pi m\left(1+t^{\delta-1}\right)$:
\begin{equation} \label{sevenfiftyone}
\sum_{n=[t]+1}^{\left[ m\left(1+t^{\delta-1}\right)\right]} \frac{1}{n^{\bar{s}}} = \frac{1}{1-\bar{s}} \left(1+t^{\delta-1}\right) ^{1-\bar{s}} m^{1-\bar{s}} + O(t^{-\sigma}), \qquad 0<\sigma<1, \quad t\to\infty.
\end{equation}
\begin{align} \label{sevenfiftytwo}
S_{B}(\sigma,t,\delta )= -\frac{i}{t} \sum_{m=l(t)}^{[t]} &\frac{m^{1-2\sigma}}{1-\frac{i(1-\sigma)}{t}} \left(  1+t^{\delta-1}  \right)^{1-\bar{s}} \nonumber \\
&+ O (t^{-\sigma}) \sum_{m=l(t)}^{[t]} m^{-s}, \qquad 0<\sigma<1, \quad t \to\infty.
\end{align}
Proceeding as with the evaluation of $S_B$ in Theorem \ref{lemma7.4} we find that
\begin{equation*}
\left| \sum_{m=l(t)}^{[t]} m^{-s}  \right| \le \frac{1}{1-\sigma} \left[  t^{1-\sigma} - \left(t - t^{\delta} +1 \right)^{1-\sigma} \right] 
=O\left(  \frac{t^{\delta}}{t^{\sigma}} \right), \quad 0<\sigma<1, \quad t \to\infty.
\end{equation*}
Similarly, the first sum of \eqref{sevenfiftytwo} is of order $t^{\delta} / t^{2\sigma}$. Thus,
\begin{equation} \label{sevenfiftythree}
S_B(\sigma,t,\delta) = O\left(  \frac{t^{\delta}}{t^{2\sigma}} \right), \quad 0<\sigma<1, \quad t \to\infty.
\end{equation}

We next consider $S_A$. Our approach is based on the application of the asymptotic formula \eqref{FR} in the inner sum of $S_A$.

Indeed, for the case that $P(t)=[m\left(1+t^{\delta-1}\right)]$, by applying \eqref{FR} in the inner sum of $S_A$ with $\frac{t}{\eta_2}=m$ and $\frac{t}{\eta_1}=m\left(1+t^{\delta-1}\right)$, we obtain
$$\eta_2=\frac{t}{m} \qquad \text{and} \qquad \eta_1=\frac{t}{m\left(1+t^{\delta-1}\right)}=\frac{t}{m}-\frac{t^\delta}{m}.$$

Observing that  $\frac{t^\delta}{m}\leq\frac{t^\delta}{t^{1-\delta}}=t^{2\delta-1}=o(1),$ and using that dist$\left(\frac{\eta_j}{2\pi},\mathbb{Z}\right)>\epsilon, \quad j=1,2,$ we obtain that $\left[\frac{\eta_1}{2\pi}\right]=\left[\frac{\eta_2}{2\pi}\right]$.

Similarly, for the case $P(t)=[t]$ we obtain $\left[\frac{\eta_1}{2\pi}\right]=\left[\frac{\eta_2}{2\pi}\right]=0$.

Thus, for the inner sum of $S_A$ the set of the summation of the rhs of \eqref{FR} is empty. Furthermore, the definition of $E(\sigma,t,\eta)$ given in \eqref{seventhirtysixa} implies that $$E(\sigma,t,\eta_j)=O\left(\left(\frac{\eta_j}{t}\right)^\sigma\right)=O\left(\frac{1}{m^\sigma}\right), \qquad j=1,2.$$
Thus,
\begin{equation} \label{sevenfortysix-f}
S_{A}(\sigma,t,\delta )=O\left(\sum_{m=\left[ t^{1-\delta}\right]}^{[t]}\frac{1}{m^{2\sigma}}\right), 
\end{equation}
which along with \eqref{sevenfiftythree} yields the estimate \eqref{est-s2-1}.

\end{proof}

\begin{remark}
For the particular case $\sigma=1/2$ the estimates of $S_2$ given by \eqref{sevenfortyfour}, \eqref{sevenfiftyeight} and \eqref{est-cr-n-3}, take the form $O\left(t^{\frac{3}{2}\delta}\right), \ O\left(t^\delta (\ln t)^3 \right)$ and $O\left(t^\delta\right)$, respectively, which for $\delta>0$ arbitrarily small, are essentially the same.

We note  that even the extensive use of the techniques appearing in Appendix B does not appear to provide an estimate better than $O\left(t^\delta\right)$, which is essentially the same with the estimate obtained via the ``rough" techniques of Lemma \ref{l5.2}, for $\delta\in(0,1)$. 

The result of Theorem \ref{thm5.3} yields the estimate $O(\ln t)$, for $\delta\in\left(0,\frac{1}{2}\right)$, which provides a significant improvement of the classical techniques on the estimate of $S_2$, when $\delta$ is not arbitrarily small. In \cite{Fs}, the sum $S_2$ is estimated via a completely different approach, and this yields the estimate  $O\left(t^{\delta-\frac{1}{2}}\ln t\right)+ O(1)$, for $\delta\in(0,1)$.
\end{remark}

\newpage

\begin{appendices}

\section{\hspace*{-5mm}(proof of \eqref{sixrem7}) }\label{rem6}
Let $\chi(s)$  be defined by \eqref{seventhirtysixd}, then it is shown in \cite{FL} that
\begin{equation}\label{sixrem2}
\chi(s)=\left(\frac{2\pi}{t}\right)^{s-\frac{1}{2}} e^{it} e^{i\frac{\pi}{4}}\left[1+O\left(\frac{1}{t}\right)\right], \qquad s=\sigma+it, \quad \sigma\in \mathbb{R},  \quad t\to \infty.
\end{equation}
Employing the well known identity
\begin{equation}\label{sixrem3}
\zeta(s)=\chi(s) \zeta(1-s), \qquad s\in \mathbb{C},
\end{equation}
with $s=\sigma-1+it,$ we find 
\begin{equation}\label{sixrem4}
\zeta(\sigma-1+it)=\chi(\sigma-1+it) \zeta(2-\sigma-it).
\end{equation}
Suppose that $0<\sigma<1$. Using the fact that $ \zeta(2-\sigma-it)$ is bounded as $t\to\infty$, as well as the asymptotic estimate \eqref{sixrem2}, equation \eqref{sixrem4} implies that
\begin{equation}\label{sixrem5}
\zeta(\sigma-1+it)=O\left(t^{\frac{3}{2}-\sigma}\right), \qquad 0<\sigma<1,\quad t\to\infty.
\end{equation}

Applying equation (3.1) of Theorem 3.1 in \cite{FL}, for $\eta=2\pi t$ we derive the following result:
\begin{align}\label{sixrem6}
\zeta(s) = & \sum_{n =1}^{[t]} n^{-s} - \frac{1}{1-s} t^{1-s} 
	\\ \nonumber
& + \frac{e^{-\frac{i\pi (1-s)}{2}}}{(2\pi)^{1-s}}\sum_{n=1}^\infty  \sum_{j=0}^{N-1} e^{-nz - it\ln{z}} \left(\frac{1}{n + \frac{it}{z}}\frac{d}{dz}\right)^j\frac{z^{-\sigma}}{n + \frac{it}{z}} \biggr|_{z = i2\pi t}
	\\ \nonumber
& + \frac{e^{\frac{i\pi (1-s)}{2}}}{(2\pi)^{1-s}}\sum_{n=2}^\infty
\sum_{j=0}^{N-1} e^{-nz - it\ln{z}} \left(\frac{1}{n + \frac{it}{z}}\frac{d}{dz}\right)^j\frac{z^{-\sigma}}{n + \frac{it}{z}} \biggr|_{z = -i2\pi t}
	\\ \nonumber
&   +O\biggl((2N +1)!! N 2^{2N} t^{-\sigma - N}  \biggr),\qquad  0 \leq \sigma \leq 1, \quad N \geq 2, \quad t \to \infty,	
\end{align}
where the error term is uniform for all $\sigma, N$ in the above ranges and the coefficients $c_k(\sigma)$ are given therein.
This equation is derived in \cite{FL} under the assumption that $0<\sigma<1$. However, it is straightforward to verify that it is also valid for $-1<\sigma<0.$
Equations \eqref{sixrem6} and \eqref{sixrem5} imply that
\begin{equation}\label{sixrem7a}
\sum_{m=1}^{[t]} \dfrac{1}{m^{\sigma-1+it}} = O\left(t^{\frac{3}{2}-\sigma}\right), \qquad 0<\sigma<1,\quad t\to\infty.
\end{equation}

\section{\hspace*{-5mm}(proof of \eqref{sevenfiftyeight})}

Letting $m_1=m$ and $m_1+m_2=n$ in the definition \eqref{sevenfortythree} of $S_2$ we find
\begin{equation} \label{sevenfortyfive-ap}
S_{2}(\sigma,t,\delta ) = S_{A}(\sigma,t,\delta ) + S_{B}(\sigma,t,\delta ), \quad 0<\sigma <1, \quad t>0,
\end{equation}
with $S_A$ and $S_B$ defined in \eqref{sevenfortysix} and \eqref{sevenfortyseven}, respectively.

The analysis in Theorem \ref{thm5.3} yields 
\begin{equation} \label{sevenfiftythree-ap}
S_B(\sigma,t,\delta) = O\left(  \frac{t^{\delta}}{t^{2\sigma}} \right), \quad 0<\sigma<1, \quad t \to\infty.
\end{equation}

We next consider $S_A$.  The derivation of this estimate consists of two parts:
\begin{itemize}
\item[I.] The first part involves the proof
\begin{equation} \label{sevenfiftysix}
\sum_{m=\left[ t^{1-\delta}\right]}^{[t]} \sum_{n=m+1}^{[t]} \frac{1}{m^{it}n^{-it}} = O \left(  t(\ln{t})^3 \right), \quad t\to \infty.
\end{equation}
\item[II.] The second part involves the partial summation technique for double sums.
\end{itemize}

\vspace*{5mm}

For the first part we first prove that
\begin{equation}\label{DS1}
\sum_{m=M}^{M'}\sum_{n=N}^{N'} \dfrac{1}{m^{it}} \dfrac{1}{n^{-it}} = O(t \ln t),
\end{equation}
with $n>m$, and $$\begin{cases}A_1\sqrt{t}<M<M'<2M<A_3 t, \\ A_2\sqrt{t}<N<N'<2N<A_4 t,\end{cases}$$ for some positive constants $\{A_j\}_1^4$.

In this connection, we divide the set of summation similarly to the division implemented in Theorem 1 of \cite{T2}, namely, in ``small" rectangles $\Delta_{p,q}$, such that 
 $$\begin{cases}  M+p \l_1\leq m \leq M+p\l_1+\l_1, \\ N+q \l_2\leq n \leq N+q\l_2+\l_2. \end{cases}$$ 
Moreover, we pick 
\begin{equation}\label{DS2}
\begin{cases} \l_1=c_1\frac{M^2}{t}, \\ \l_2=c_2\frac{N^2}{t}, \end{cases}
\end{equation} for some positive constants $c_1$ and $c_2$.

We make the following observations:
\begin{itemize}
\item $\displaystyle \begin{cases} 1 \leq \l_1 \leq M \Leftrightarrow  A_1\sqrt{t}<M<A_3 t, \\ 1 \leq \l_2 \leq N \Leftrightarrow  A_2\sqrt{t}<N<A_4 t,\end{cases}$ for some positive constants $\{A_j\}_1^4$.
\item The number of the ``small" rectangles $\Delta_{p,q}$  is $O\left(\dfrac{MN}{\l_1 \l_2}\right)$.
\end{itemize}
We use Theorem 2.16 of \cite{K} with $$f(x,y)=t(\ln x - \ln y).$$
Then, in each rectangle $\Delta_{p,q}$, with $n>m$ (equivalently $x>y$), the conditions of this theorem are satisfied with $\lambda_1=\frac{t}{M^2}$ and $\lambda_2=\frac{t}{N^2}$, because $$\big| f_{xx} \big| =\frac{t}{x^2}, \quad \big| f_{yy} \big| =\frac{t}{y^2} \quad \text{ and } \ \big| f_{xy} \big| =0.$$
Using the following facts: 
\begin{itemize}
\item the conditions  $\displaystyle \begin{cases} M > A_1 \sqrt{t}, \\ N> A_2 \sqrt{t}, \end{cases}$ imply that $\displaystyle \begin{cases} \lambda_1 <\frac{1}{A_1^2}, \\ \lambda_2 <\frac{1}{A_2^2} \end{cases}$, 
\item all the quantities $ \ln \big|\Delta_{p,q} \big|, \ \big|\ln \lambda_1 \big|$ and $\big|\ln \lambda_2 \big|$ are of order $O(\ln t)$,
\end{itemize}
and employing equation (2.56) of \cite{K}, we find
\begin{equation}\label{DS3}
\mathop{\sum\sum}_{(m,n)\in \Delta_{p,q}} e^{if(m,n)} = O\left(\dfrac{\ln t}{\frac{t}{MN}}\right).
\end{equation}
Thus, the fact that the number of the rectangles $\Delta_{p,q}$  is $O\left(\dfrac{MN}{\l_1 \l_2}\right)$, implies that
\begin{equation}\label{DS4}
\sum_{m=M}^{M'}\sum_{n=N}^{N'} e^{if(m,n)} = O\left(\dfrac{MN}{t} \ln t \dfrac{MN}{\l_1 \l_2}\right).
\end{equation}
Equation \eqref{DS1} follows from applying \eqref{DS2} in \eqref{DS4}.

Finally, using the classical splitting for the sets of summation for exponential sums, see \cite{T} and \cite{T2}, equation \eqref{sevenfiftysix} follows from applying \eqref{DS1} for  $O\left(\left(\ln t^\delta\right)^2\right)=O\left(\left(\delta\ln t\right)^2\right)=O\left((\ln t)^2\right)$ times.

Considering the second part, under the condition that the expressions 
\begin{equation}\label{psks}
b_{m,n} - b_{m+1,n}, \quad b_{m,n}-b_{m,n+1}, \quad b_{m,n}-b_{m+1,n}-b_{m,n+1}+b_{m+1,n+1},\end{equation}  
keep their sign, the following result is derived in \cite{T2}:
\begin{equation} \label{sevenfiftyfive1}
\left|  \sum_{m=1}^{M} \sum_{n=1}^{N} a_{m,n} b_{m,n} \right| \le 5GH,
\end{equation}
where
\begin{equation} \label{sevenfiftysix1}
S_{m,n} \Doteq \sum_{\mu=1}^{m} \sum_{\nu=1}^{n} a_{\mu,\nu}, \quad |S_{m,n}|\le G, \quad 1\le m\le M, \quad 1\le n\le N,
\end{equation}
with
\begin{equation} \label{sevenfiftyseven}
b_{m,n} \in \mathbb{R}, \quad 0 \le b_{m,n} \le H.
\end{equation}
We apply the above argument for $\displaystyle b_{m,n}=\frac{1}{m^\sigma n^\sigma}$,
thus  the expressions in \eqref{psks} keep their sign, and furthermore $$H=\dfrac{1}{t^{(1-\delta)\sigma}}\dfrac{1}{t^{(1-\delta)\sigma}}=t^{-2\sigma}t^{2\delta\sigma}.$$

Combining the above result with \eqref{sevenfiftysix} yields 
\begin{equation} \label{sevenfiftyeighta}
S_A(\sigma,t,\delta) = O \left( t^{1-2\sigma} t^{2\delta\sigma} (\ln{t})^3 \right), \quad 0<\sigma<1, \quad t \to\infty.
\end{equation}


Equations \eqref{sevenfortyfive-ap}, \eqref{sevenfiftythree-ap}, \eqref{sevenfiftyeighta} imply \eqref{sevenfiftyeight}.

\end{appendices}

\section*{Acknowledgments}

Both authors are supported by the EPSRC, UK.


\vspace{5mm}


\begin{thebibliography}{999}

\bibitem[B]{B} J. Bourgain, Decoupling, exponential sums and the Riemann zeta function, J. Am. Math. Soc., 2017.

\bibitem[F]{Fs} A. S. Fokas, A novel approach to the Lindel{\"o}f hypothesis, arXiv preprint, arXiv:1708.06607, 2018.



\bibitem[FL]{FL} A.S. Fokas and J. Lenells, On the asymptotics to all orders of the Riemann Zeta function and of a two-parameter generalization of the Riemann Zeta function, Memoirs of AMS, to appear.

\bibitem[IM]{IM} H. Ishikawa and K. Matsumoto, On the estimation of the order of Euler-Zagier multiple zeta-functions, Illinois J. Math., 2003.

\bibitem[K]{K} E. Kr{\"a}tzel, Lattice points, Springer, 1989.

\bibitem[KT]{KT} I. Kiuci and Y. Tanigawa, Bounds for double zeta functions, Ann. Scuola Norm. Sup. Pisa Cl. Sci., 2006.

\bibitem[T]{T} E.C. Titchmarsch, The theory of the Riemann Zeta-function, Oxford University Press, 2$^\text{nd}$ ed., 1987.

\bibitem[T2]{T2} E.C. Titchmarsch, On Epstein's Zeta-function, Proc. London Math. Soc., 1934.









\end{thebibliography}
\end{document}